\documentclass[10pt]{amsart}
\newtheorem{thm}{Theorem}[section]
\newtheorem{cor}[thm]{Corollary}

\newtheorem{lem}[thm]{Lemma}
\newtheorem{prop}[thm]{Proposition}
\theoremstyle{definition}
\newtheorem{defn}[thm]{Definition}
\theoremstyle{remark}
\newtheorem{rem}[thm]{Remark}
\numberwithin{equation}{section}
\newtheorem{ex}[thm]{Example}
\newcommand{\E}{\mathcal{E}}

\newcommand{\ds}{\,{\rm d}s}
\newcommand{\beq}{\begin{equation}}
\newcommand{\eeq}{\end{equation}}

\def\rL{{\rm L}}
\def\rd{{\rm d}}
\begin{document}

\title[]{Hypercontractivity, Nash inequalities and subordination\\ for classes of nonlinear semigroups}%
\author{Fabio Cipriani}
\address{Politecnico di Milano, Dipartimento di Matematica,
piazza Leonardo da Vinci 32, 20133 Milano, Italy.}%
\email{fabio.cipriani@polimi.it}%
\author{Gabriele Grillo}
\address{Dipartimento di Matematica, Politecnico di Torino, Corso Duca degli Abruzzi 24, 10129 Torino, Italy}
\email{gabriele.grillo@polito.it}
\begin{abstract}
A suitable notion of hypercontractivity for a nonlinear semigroup $\{T_t\}$ is shown to imply
Nash--type inequalities for its generator $H$, provided
a subhomogeneity property holds for the energy
functional $(u,Hu)$. We use this fact to prove that, for semigroups generated by operators of
$p$-Laplacian-type, hypercontractivity implies ultracontractivity. Then we introduce the notion of
subordinated nonlinear semigroups when the corresponding Bernstein
function is $f(x)=x^\alpha$, and write an explicit formula
for the associated generator. It is shown that hypercontractivity still holds for the subordinated semigroup and, hence, that
Nash-type inequalities hold as well for the subordinated generator.
\end{abstract}
\maketitle

\section{Introduction}
Since the seminal papers of Gross \cite{Gr}, Nelson \cite{Ne},
Federbush \cite{Fe}, Simon and H\o egh-Krohn \cite{SiH}, Davies
and Simon \cite{DS}, the relations among various types of
contractivity properties of linear semigroups on the one hand and
functional inequalities satisfied by their generators on the other
hand, have been intensively investigated. In particular the notions
of hypercontractivity, supercontractivity and ultracontractivity
for {\it linear} Markov semigroups have been related to
logarithmic Sobolev inequalities and/or Sobolev, Nash and
Gagliardo-Nirenberg inequalities involving Dirichlet forms. See
for example \cite{Ba}, \cite{D}, \cite{Gr2} for comprehensive
discussions.

Moreover, generalizations of functional inequalities of Nash and
Gagliardo-Nirenberg type have been considered in \cite{BCLS}, and
in particular it is shown there that, under suitable assumptions,
the validity of a single Gagliardo-Nirenberg inequality implies
the validity of a whole class of them.

Recently, for some classes of nonlinear parabolic partial
differential equations, contractivity properties of their
solutions have been proved as a consequence of the validity of
suitable logarithmic Sobolev inequalities involving the {\it
nonlinear Dirichlet forms} (in the sense of \cite{CG}) associated
to their generator. See \cite{CG2} and \cite{CG3} for the
evolution equation driven by the $p$-Laplacian and \cite{BoG} for
the porous media equation. We refer to the recent work of J.L. Vazquez \cite{Va} for an excellent discussion
of the known smoothing and decay properties for classes of evolutions including the porous media equation and the evolution driven by the $p$-Laplacian, in the Euclidean case.

One of the aims of the present paper is to continue to investigate such
relations in the nonlinear setting. In particular we shall
concentrate ourselves on a sort of converse of what has been
investigated in \cite{CG2}, \cite{CG3} and \cite{BoG}, namely on
the consequences which can be drawn assuming that a nonlinear
semigroup is, in a sense to be defined later, hypercontractive. In
the linear case hypercontractivity is equivalent to a logarithmic
Sobolev inequality for the Dirichlet form associated to the
generator, but in general it does {\it not} imply a Nash (or a
Sobolev) inequality, as the Ornstein-Uhlenbeck example shows.

Concerning the methods, our starting point will be the ideas of
Gross \cite{Gr}. However, nonlinearity plays here a special role
so that new phenomena occur: in particular it will be shown that
nonlinear hypercontractivity implies functional inequalities of
Nash-type for the generator $H$ provided it
satisfies the {\it subhomogeneity property}
\[
(\lambda u, H(\lambda u))\le M\lambda^p(u,Hu).
\]
for all positive $\lambda$, all $u$ in the L$^2$ domain of $H$, a
suitable $M>0$ and a suitable $p>2$, where $(\cdot,\cdot)$ is the
scalar product in L$^2$. Then we draw another surprising
consequence of this fact. Consider the evolution equation driven
by the subgradient of the functional \[ \E_p(u):=\int_M|\nabla
u|^p\,{\rm d}m, \] where $(M,g)$ is a complete Riemannian
manifold, $\nabla$ is the Riemannian gradient and $m$ is a
$\sigma$-finite nonnegative measure on $M$ with the property that
$\nabla$ is closable as an operator from $\rL^2(M,m)$ into
$\rL^2(TM,m)$ (notice that $m$ need not be the Riemannian measure). We call this operator a {\it generalized
Riemannian $p$-Laplacian}. We shall show that, when $p>2$, its
associated semigroup is ultracontractive whenever it is
hypercontractive, a property which has of course no linear
analogue (i.e. no analogue in the case $p=2$).

The second main goal of the paper is to introduce the process of {\it subordination} of a given nonlinear semigroup $\{T_t\}$ w.r.t.
convolution semigroups of probability measures (see e.g. \cite{J}). In the present paper we shall
deal only with the subordination associated to the Bernstein
functions $f(x)=x^\alpha$, $\alpha\in(0,1]$, a choice for which we
give an explicit description of the subordinated nonlinear
generator in terms of the original semigroup. This procedure extends the classical Bochner's one \cite{Bo} when
the starting semigroup is linear.
Then we shall show
that under some assumptions, satisfied in the case in which $\{T_t\}$ is the semigroup associated to a generalized $p$-Laplacian and hypercontractivity holds for $\{T_t\}$, the subordinated semigroup is hypercontractive and subhomogeneous and, hence, Nash-type inequalities hold for its generator as well.

The plan of the paper is as follows. In section 2 we prove
Nash-type inequalities for nonlinear
hypercontractive (or supercontractive) semigroups which are also subhomogeneous: see Theorem \ref{GNI1}.  In
section 3 we prove that if the semigroup associated to a generalized Riemannian $p$-Laplacian
is hypercontractive it is also ultracontractive: see Theorem \ref{plap}. Section 4 is
devoted to the construction of nonlinear subordinated semigroups. The main result there is an explicit formula for the right derivative $A_\alpha u$ at $t=0$ of what we call the {\it nonlinear subordinated family} $S_tu:=\int_0^{+\infty}T_su\,\mu^{(\alpha)}_t(\ds)$ for a certain choice of the subordinator $\mu^{(\alpha)}_t$, and the proof that the resulting operator is monotone, thus giving rise to a well defined nonlinear (subordinated) strongly continuous, nonexpansive semigroup: see Theorem \ref{potenza}.
In section 5 we prove Nash-type inequalities for the subordinated
generators associated to the Bernstein function $f(x)=x^\alpha$,
$\alpha\in(0,1]$ and to the semigroup driven by a generalized
$p$-Laplacian, provided such semigroup is hypercontractive: see Theorem \ref{sub}.

We thank the referee for his (or her) very careful reading of the first version of the present paper.

\section{Nonlinear hypercontractive and supercontractive semigroups}
In the present section we shall deal with {\it nonlinear
hypercontractive semigroups}. Before giving the appropriate definition, we recall that a (nonlinear) strongly continuous semigroup $\{T_t\}_{t\ge0}$ on a a Hilbert space L$^2(X,m)$ is said to be \it nonexpansive \rm if, for all $u,v\in{\rm L}^2(X,m)$, $t\ge0$, one has $\|T_tu-T_tv\|_2\le\|u-v\|_2$. This is in general different from requiring that $\{T_t\}_{t\ge0}$ is \it contractive, \rm namely that $\|T_tu\|_2\le \|u\|_2$ for all $u\in{\rm L}^2(X,m)$, $t\ge0$.
\begin{defn}\label{iper}
A strongly continuous contractive semigroup $\{T_t\}_{t\ge0}$,
not necessarily linear, on a Hilbert space ${\rm L}^2(X,m)$ (see
\cite{B}, \cite{S}) is said to be {\it hypercontractive} if there
exist $\varepsilon>0$ and continuously differentiable functions
$r:[0,\varepsilon)\to [2,+\infty)$, $\alpha, k
:[0,\varepsilon)\to(0,+\infty)$ with $r(0)=2$, $\dot r(t)>0$ for
all $t$, $\alpha(0)=1$, $k(0)=1$ and such that, for all $u\in{\rm
L}^2(X,m)$ and all $t\in[0,\varepsilon)$ one has
\beq\label{iperformula} \Vert T_tu\Vert_{r(t)}\le k(t)\Vert u
\Vert_2^{\alpha(t)}. \eeq
\end{defn}

We shall also use the following definition.

\begin{defn}\label{assum}
A strongly continuous contractive semigroup $\{T_t\}_{t\ge0}$,
not necessarily linear, on a Hilbert space ${\rm L}^2(X,m)$
is said to be $(\beta,s)$-\it supercontractive \rm if there exists $\beta>0$, $s>2$ and a
continuous function $k:(0,+\infty)\to [0,+\infty)$ such that $T_tu\in\rL^s(X,m)$ for all
$t>0$ and all $u\in\rL^2(X,m)$ and
\begin{equation}
\label{super} \Vert T_tu\Vert_s\le k(t)\Vert u\Vert_2^\beta\ \ \
\forall t>0, \forall u\in\rL^2(X,m).
\end{equation}
\end{defn}

\begin{ex}
Bounds of the form \eqref{iperformula}, \eqref{super} hold true, for example, for
the solutions to the Euclidean $p$-heat equation $\dot
u=\triangle_pu:=\nabla\cdot(|\nabla u|^{p-2}\nabla u)$ ($p>2$) or
for the porous media equation $\dot
u=\triangle(u^m):=\triangle(|u|^{m-1}u)$ ($m>1$). See \cite{CG2, BoG, BoG2}.
\end{ex}

\begin{rem}
Hereafter we shall denote by $H$ the principal section, in the
sense of Brezis \cite{B}, of the generator of the semigroup
$\{T_t\}_{t\ge0}$. \end{rem}

Our first result consists in remarking that an immediate
consequence of the definition of
hypercontractivity is the validity of a suitable {\it
inhomogeneous} logarithmic Sobolev inequality involving the
functional $(u,Hu)$.

\begin{lem}\label{iperteo}
Let $\{T_t\}_{t\ge0}$ be a not necessarily linear hypercontractive
semigroup in the sense of Definition \ref{iper}. Then the
logarithmic Sobolev inequality \beq\label{logsob} \int_Xu^2\log
|u|\rd m-c_1\|u\|_2^2\log\|u\|_2\le c_2(u,Hu)_{\rL^2}+c_3\|u\|_2^2
\eeq holds for any $u$ belonging to the $\rL^2$ domain of the
$\rL^2$ generator $H$ of $\{T_t\}_{t\ge0}$, where
\[
c_1=\dfrac{\dot r(0)+2\dot\alpha(0)}{\dot r(0)},\ \
c_2=\dfrac{1}{\dot r(0)}\ \  c_3=\dfrac{2\dot k(0)}{\dot r(0)}.
\]
\end{lem}
\begin{proof}
By \cite[Lemma 3.8]{Gr2} we have that, if $r$ is a continuously differentiable function with values in
$[2,+\infty)$ and $r(0)=2$.
\[
\left.\dfrac{\rd}{\rd t}\Vert
T_tu\Vert_{r(t)}\right|_{t=0}=\|u\|^{-1}_{2}\left(\dfrac{\dot r(0)}{2}\int_X
u^2\log\dfrac{|u|}{\|u\|_2}\rd m-(u,Hu)_{{\rm L}^2}\right).
\]
\end{proof}

Notice that $c_1=1$ if and only if
$\dot\alpha(0)=0$. If this is the case the log-Sobolev inequality \eqref{logsob} involves the usual entropy functional.

\begin{defn}\label{homog}
We say that the generator $H$ of the semigroup $\{T_t\}_{t\ge0}$
is {\it subhomogeneous} of degree $p>0$ if $(u,Hu)_{\rL^2}\ge0$ for all $u\in D(H)$ and there exists a positive
$M$ such that, for all positive $\lambda$ and all $u\in{\rm
Dom}\,H$ one has that $\lambda u\in{\rm Dom}\,H$ as well and
moreover
\[
( \lambda u, H(\lambda u))_{\rL^2}\le M\lambda^p(u,Hu)_{\rL^2}.
\]
\end{defn}

\begin{ex}
Consider a domain $D$ in the Euclidean space ${\mathbb R}^n$ and the
operator $H$ defined on a suitable subset of L$^2(D)$ and
given, when the quantity below exists, by
\[
Hu=-\nabla\cdot \left(a(x,u,\nabla u)|\nabla u|^{p-2}\nabla
u\right),
\]
$a:D\times{\mathbb R}\times{\mathbb R}^n\to[0,+\infty)$ being a positive differentiable function. Let us suppose that $H$ can be extended as a densely defined
maximally monotone operator: for conditions guaranteeing this
facts see \cite{S}. The subhomogeneity assumption of degree $p$
is equivalent to:
\[
a(x,\lambda v,\lambda\xi)\ge Ma(x,v,\xi)\ \ \forall v\in{\mathbb
R}, \xi\in{\mathbb R}^n,\lambda>0
\]
and for a suitable fixed positive $M$. Of course this implies a
similar {\it lower bound} on $a(x,v,\xi)$, so that there exists $C>0$ such that
\[
\frac1C\, a(x,\lambda v,\lambda \xi)\le a(x,v,\xi)\le Ca(x,\lambda v,\lambda\xi),\ \ \forall \lambda>0,x\in D,v\in{\mathbb R}, \xi\in{\mathbb R}^n.
\] Letting $\lambda\to 0$ and using the continuity of $a$ yields that upper
and lower bounds on $a(x,v,\xi)$ in terms of functions of the
space variable $x$ only must hold.
\end{ex}

\begin{lem}\label{logarit}
Let $\{T_t\}_{t\ge0}$ be a nonlinear semigroup which is
hypercontractive in the sense of Definition \ref{iper}. Assume moreover
that $H$ is strictly positive in the sense that
\[
(u,Hu)_{\rL^2}>0 \ \ \forall u\in{\rm Dom}\,H, u\not\equiv0,
\]
that $H$ is subhomogeneous of degree $p$ in the sense of
Definition \ref{homog} for a suitable $p>2$, and that $\dot \alpha(0)<0$,
$\dot r(0)>0$. Then the logarithmic Sobolev inequality
\beq\label{log} \int_X\dfrac{u^2}{\Vert u\Vert_2^2}\log
\dfrac{u^2}{\Vert u\Vert_2^2}\rd m\le
k_1\log\left(k_2\dfrac{(u,Hu)_{\rL^2}}{\Vert
u\Vert_2^p}\right)\eeq holds true for any $u$ in the $\rL^2$
domain of $H$, where the positive constants $k_1,k_2$ are defined
by
\[
k_1=-\dfrac{4\dot\alpha(0)}{(p-2)\dot r(0)},\ \
k_2=-\dfrac{M(p-2)}{2\dot\alpha(0)}e^{1-[\dot
k(0)(p-2)]/[\dot\alpha(0)]}.
\]
If in addition $X$ has finite measure, there exists $C>0$ such that the coercive bound
\[
\|Hu\|_2\ge C\|u\|_2^{p-1}.
\]
holds true.
\end{lem}

\begin{proof}
Inequality (\ref{logsob}), where $u$ is replaced by
$\lambda u$ with $\|u\|_2=1$, $\lambda>0$, implies by the
subhomogeneity of $H$:
\[
\int_Xu^2\log|u|\rd m+\widetilde c_1\log\lambda\le
Mc_2\lambda^{p-2}(u,Hu)_{\rL^2}+c_3
\]
where $\widetilde c_1:=1-c_1=-2\dfrac{\dot\alpha(0)}{\dot r(0)}>0.$
We rewrite the above formula as \beq\label{parapapa} \widetilde
c_1\log \lambda-Mc_2\lambda^{p-2}(u,Hu)_{\rL^2}\le
c_3-\int_Xu^2\log |u|\rd m. \eeq The real function
$a(\lambda):=A\log\lambda-B\lambda^{p-2}$ ($\lambda>0$, $A,B>0$)
attains its minimum when $
\lambda=\overline\lambda:=\left[A/[B(p-2)]\right]^{1/(p-2)}
$
and one has
\[
a(\overline\lambda)=\dfrac{A}{p-2}\log\left(\dfrac{A}{Be(p-2)}\right).
\]
Then
\[
\dfrac{\widetilde c_1}{p-2}\log\left(\dfrac{\widetilde
c_1}{Mc_2(u,Hu)_{\rL^2}e(p-2)}\right)\le c_3-\int_Xu^2\log|u|\rd m
\]
which can be rewritten as
\[
\begin{split}
\int_X u^2\log |u|\rd m&\le c_3+\dfrac{\widetilde
c_1}{p-2}\log\left[\dfrac{Mc_2}{\widetilde
c_1}e(p-2)(u,Hu)_{\rL^2}\right]\\
&=\dfrac{\widetilde c_1}{p-2}\log\left[e^{c_3(p-2)/\widetilde c_1
}\dfrac{Mc_2}{\widetilde
c_1}e(p-2)(u,Hu)_{\rL^2}\right]\\
&=\dfrac{k_1}{2}\log[k_2(u,Hu)_{\rL^2}]\end{split}
\]
with $k_1,k_2$ as in the statement.

Writing now the latter inequality with $u$ replaced by $u/\|u\|_2$
and using again subhomogeneity yields \eqref{log}.

As for the last statement, just notice that under the running assumption the bound
\[\int_X\dfrac{u^2}{\Vert u\Vert_2^2}\log
\dfrac{u^2}{\Vert u\Vert_2^2}\rd m\ge -\log[m(X)]\]
holds true by Jensen's inequality. Therefore
\[\log\left(k_2\dfrac{(u,Hu)_{\rL^2}}{\Vert
u\Vert_2^p}\right)\ge-\frac{\log m(X)}{k_1}=\log\left[\left(m(X)\right)^{-1/k_1}\right]\]
so that \[(u,Hu)_{\rL^2}\ge \frac1{k_2m(X)^{1/k_1}}\|u\|_2^p.\]
The final statement follows because $(u,Hu)_{\rL^2}\le |(u,Hu)_{\rL^2}|\le \|u\|_2\|Hu\|_2$.

\end{proof}

We now prove that a Nash-type inequality follows
straightly from the logarithmic Sobolev inequality given in Lemma
\ref{logarit}.
We shall use this result in the final section, in
which Nash-type inequalities will be proved for operators which are the generators of
semigroups which are, in a sense to be defined, subordinated to the $p$-Laplacian semigroup.

We first define the Young functional \beq\label{young}
J(p,u)=\int_X \frac{|u(x)|^p}{\|u\|_p^p}
\log\left[\frac{|u(x)|}{\|u\|_p}\right]\rd m(x).
\eeq Notice that, by definition, $J(r,u)=\dfrac{1}{r}J(1,|u|^r)$.

\begin{thm}\label{GNI1}(Nash-type inequalities). Let $\{T_t\}_{t\ge0}$ be a nonlinear semigroup which is
hypercontractive in the sense of Definition \ref{iper}. Assume moreover
that $H$ is strictly positive in the sense that
\[
(u,Hu)_{\rL^2}>0 \ \ \forall u\in{\rm Dom}\,H, u\not\equiv0,
\]
that $H$ is subhomogeneous of degree $p$ in the sense of
Definition \ref{homog} for a suitable $p>2$, and that $\dot \alpha(0)<0$,
$\dot r(0)>0$. Then, for any fixed $m\in[1,2)$ the inequality \beq\label{m-nash}
\|f\|_2\le C(Hf,f)_{{\rm
L}^2}^{\frac{k_1(2-m)}{m+pk_1(2-m)}}\|f\|_m^{\frac{m}{m+pk_1(2-m)}}
\eeq holds true for a suitable constant $C$ and for all $f\in{\rm
Dom}\,H\cap{\rm L}^m$, where $k_1=-\dfrac{4\dot\alpha(0)}{(p-2)\dot r(0)}>0$ is the constant appearing in
Lemma \ref{logarit}. In particular the Nash inequality \beq\label{1-nash}
\|f\|_2\le C(Hf,f)_{{\rm
L}^2}^{\frac{k_1}{1+pk_1}}\|f\|_1^{\frac{1}{1+pk_1}} \eeq holds
true for all $f\in{\rm Dom}\,H\cap{\rm L}^1$.

A similar conclusion holds if, keeping fixed the assumptions on $H$, $\dot \alpha(0)$ and $\dot r(0)$, one replaces the hypercontractivity assumption by the assumption that the semigroup considered is
$(\beta,s)$-supercontractive in the sense of Definition
\ref{assum}, for suitable $\beta<1$, $s\ge\max(2,2/\beta)$, and with $k$ such that
$\lim_{t\to0}k(t)^t$ is finite. In fact, the
inequalities \eqref{m-nash} and \eqref{1-nash} hold if one replaces $k_1$ above with $k_1=2(1-\beta)/([\beta(p-2)(s-2)]$.

\end{thm}
\begin{proof}
We rewrite the assertion of Lemma \ref{logarit} as
\[
J(1,f^2)\le k_1\log\left[k_2\dfrac{(Hf,f)_{{\rm
L}^2}}{\Vert f \Vert_2^p}\right].
\]
Next we recall that, if $q>r$, \beq \dfrac{\Vert u
\Vert_q}{\Vert u\Vert_r}\le e^{\frac{q-r}{qr}J(1,|u|^q)}. \eeq
See \cite{BG4} for a proof of this fact. By using such inequality with $q=2$, $r=m<2$ we get
\[
\log\dfrac{\Vert f\Vert_2}{\Vert f\Vert_m}\le
\log\left\{\left[k_2\dfrac{(Hf,f)_{{\rm L}^2}}{\Vert f
\Vert_2^p}\right]^{k_1(2-m)/m}\right\}
\]
which is an equivalent form of our statement in the hypercontractive case.

To deal with the supercontractive case,
first we notice that there exists a positive $C$
such that the logarithmic Sobolev inequality
\begin{equation}
\dfrac{s-2}{s}\int_Xu^2\log |u|\rd m-\dfrac{s\beta-2}{s\beta}\Vert
u\Vert_{{2}}^2\log \Vert u\Vert_{{2}}\le
(u,Hu)_{\rL^2}+C\Vert u\Vert_{{2}}^2
\end{equation}
holds true for all $u\in\rL^2\cap{\rm Dom}\,H$, where $H$ is the
generator of $T_t$ and ${\rm Dom}\,H$ is its $\rL^2$ domain. In fact, an elementary interpolation argument with the exponents $\phi=\frac{\vartheta\beta}{\vartheta\beta+1-\vartheta}$, $1-\phi=\frac{1-\vartheta}{\vartheta\beta+1-\vartheta}$ shows that, for all $\vartheta\in(0,1)$:
\begin{equation}\label{int1}
\Vert
T_tu\Vert_{\frac{2s(\vartheta\beta+1-\vartheta)}{2(1-\vartheta)+s\vartheta\beta}}\le
k(t)^{(1-\vartheta)/[\vartheta\beta+1-\vartheta]}\Vert
u\Vert_2^{\beta/[\vartheta\beta+1-\vartheta]},
\end{equation}
Choosing
$1-\vartheta=t$ for $t$ small we get:
\[
\Vert T_tu\Vert_{p(t)}^{p(t)a(t)}\le
k(t)^{\frac{2t}{\beta}}\Vert u\Vert_{2}^2
\]
with
\[
p(t)=\dfrac{2s((1-t)\beta+t)}{2t+s\beta(1-t)};\ \ \
a(t)=\dfrac{2t+s\beta(1-t)}{s\beta}.
\]
By the assumption on the behaviour of the function $k$ near $t=0$ we have also
\begin{equation}\label{tecnica}
\Vert T_tu\Vert_{p(t)}^{p(t)a(t)}\le
A\Vert u\Vert_{2}^2
\end{equation}
for a suitable constant $A$, where we can assume that $A>1$ (note the strict inequality: if the function $k$ is such that the above inequality holds with $A=1$ the argument below can be simplified proceeding as in \cite{Gr}).
This is a hypercontractive bound similar to the ones studied above. In the l.h.s. of \eqref{tecnica} we have a real function $f$ of $t\ge0$ which is differentiable at $t=0$ and satisfies $f(0)< A\|u\|_2^2$. Then the tangent line to the graph of $f$ at $t=0$ lies
below $A\|u\|_2^2$ for sufficiently small $t$
Next, \cite[Lemma 3.8]{Gr2} (see the proof of Lemma \ref{iperteo}) and the chain rule imply that
\[
\left.\frac{{\rm d}}{{\rm d}t}\Vert T_tu\Vert_{p(t)}^{p(t)a(t)}\right|_{t=0}
=\dfrac{2(s-2)}{s}\int_Xu^2\log|u|\rd m-2\int_Xu(Hu)\,\rd
m-2\dfrac{s\beta-2}{s\beta}\Vert u\Vert_{2}^{2}\log \Vert
u\Vert_{\rL^{2}}.
\]
Therefore, for $t$ sufficiently small,
\[
\|u\|_2^2+t\left[\dfrac{2(s-2)}{s}\int_Xu^2\log|u|\rd m-2\int_Xu(Hu)\,\rd
m-2\dfrac{s\beta-2}{s\beta}\Vert u\Vert_{2}^{2}\log \Vert
u\Vert_{\rL^{2}}\right]\le A\|u\|_2^2
\]
or equivalently
\[
\dfrac{2(s-2)}{s}\int_Xu^2\log|u|\rd m-2\int_Xu(Hu)\,\rd
m-2\dfrac{s\beta-2}{s\beta}\Vert u\Vert_{2}^{2}\log \Vert
u\Vert_{\rL^{2}}\le \frac{A-1}{t}\|u\|_2^2
\]
again for $t$ sufficiently small. Finally, for the same values of $t$,
\[
\int_Xu^2\log |u|\rd m-\dfrac{s\beta-2}{\beta(s-2)}\Vert
u\Vert_{{2}}^2\log \Vert u\Vert_{{2}}\le
\dfrac{s}{s-2}(u,Hu)_{\rL^2}+\dfrac{s(A-1)}{2t(s-2)}\Vert
u\Vert_{{2}}^2.
\]
Fixing $t$ sufficiently small we find that there is a constant $C>0$ such that
\[
\int_Xu^2\log |u|\rd m-\dfrac{s\beta-2}{\beta(s-2)}\Vert
u\Vert_{{2}}^2\log \Vert u\Vert_{{2}}\le
\dfrac{s}{s-2}(u,Hu)_{\rL^2}+C\Vert
u\Vert_{{2}}^2.
\]

To proceed further we use the same reasoning used in the proof of Lemma \ref{logarit}. In fact we can proceed exactly as in that proof provided the constant $(s\beta-2)/[\beta(s-2)]$, which
takes here the role of $c_1$ there, lies in the interval (0,1). This is
true under our assumptions on $s$ and $\beta$. Then we get that the logarithmic Sobolev inequality \beq\label{log2}
\int_X\dfrac{u^2}{\Vert u\Vert_2^2}\log \dfrac{u^2}{\Vert
u\Vert_2^2}\rd m\le k_1\log\left(k_2\dfrac{(u,Hu)_{\rL^2}}{\Vert
u\Vert_2^p}\right)\eeq holds true for any $u$ in the $\rL^2$
domain of $H$, where $k_2$ is a suitable positive constant and
$k_1=2(1-\beta)/([\beta(p-2)(s-2)]$. The final statement follows as in the hypercontractive case.
\end{proof}

\begin{rem}
The restriction on $s$ in the result concerning supercontractive semigroups may seem strange since a similar result is proved in the same Theorem under the assumption of hypercontractivity only. However it is related to the additional request $\dot\alpha(0)<0$ assumed in the hypercontractive case.
\end{rem}

\begin{rem}
A whole family of Gagliardo-Nirenberg inequalities follows from the Nash inequality proved above if the functional $W(u):=(u,Hu)_{{\rm L}^2}^{1/p}$ satisfies certain contractivity properties defined in \cite{BCLS}. This happens in particular in the case in which $H$ is the $p$-Laplacian on a Riemannian manifold, a situation which will be discussed further in the present paper.
\end{rem}

\section{From hypercontractivity to ultracontractivity for evolutions driven by
generalized $p$-Laplacians}

In this section we shall deal with a particular choice of the
semigroup $\{T_t\}_{t\ge0}$. In fact, let $(M,g)$ be a complete
Riemannian manifold, whose Riemannian gradient is indicated by
$\nabla$. Then we choose a $\sigma$-finite nonnegative measure
$m$ on $M$, and assume that $\nabla$ is a closed operator from
L$^2(M,m)$ to L$^2(TM,m)$. Notice that $m$ {\it need not be the Riemannian measure}. We shall then consider the strongly
continuous contraction semigroup $\{T_t\}_{t\ge0}$ associated to
the subgradient of the convex lower semicontinuous functional given by
\beq\label{pl} \E_p(u):=\int_M|\nabla u|^p\,{\rm d}m \eeq for
those $\rL^2(X,m)$ functions for which the integral is finite, and
by $+\infty$ otherwise in $\rL^2(M,m)$. In fact, lower
semicontinuity of the above functional is a consequence of the
assumption that $\nabla$ is closed (see \cite{CG4}), but could be
alternatively assumed directly.

Hereafter the condition $p>2$ will be necessary to apply the
previous results. The generator of the semigroup that we consider in this section is the operator given formally by $H=-\nabla^*(|\nabla
u|^{p-2}\nabla u)$, where $\nabla^*$ is for the formal adjoint of the gradient with
respect to the inner products of $\rL^2(M,m)$ and of $\rL^2(TM,m)$.

We stress that we choose the above semigroup as a {\it model case}
for our discussion to hold, but that much more general situations
can be dealt with by identical methods: see \cite{CG, CG4} for details.

The semigroup associated to the subgradient of $\E_p$ may or may
not be hypercontractive. Our aim here is to show some consequences
of hypercontractivity, if it holds: more precisely we shall show
that under the same assumptions which allow to prove a homogeneous
logarithmic Sobolev inequality of the type given in Lemma
\ref{logarit}, ultracontractivity of the semigroup hold as well.

Before turning to this our first comment is that, as a consequence
of our previous results, logarithmic Sobolev inequalities must hold.

\begin{cor}\label{plapl}
Let $\{T_t\}_{t\ge0}$ be the semigroup associated to the
functional $\E$ given in (\ref{pl}), with $p>2$. Assume that it is
hypercontractive in the sense of Definition \ref{iper}, with $\dot
\alpha(0)<0$, $\dot r(0)>0$ and $\dot r(0)+2\dot\alpha(0)>0$. Let the dimension $d>0$ of the hypercontractive semigroup considered be defined by
\[
d=\dfrac{-4\dot\alpha(0)p}{(p-2)[\dot r(0)+2\dot\alpha(0)]},
\]
Then the logarithmic Sobolev
inequality
\beq\label{log4} \int_X\dfrac{u^p}{\Vert
u\Vert_p^p}\log \dfrac{u^p}{\Vert u\Vert_p^p}\rd m\le
\dfrac{d}{p}\log\left(K\dfrac{\E_p(u)}{\Vert
u\Vert_p^p}\right)\eeq holds true for a suitable constant $K$ and
for any $u\in\rL^2(X,m)$.
\end{cor}

\begin{proof} It suffices to prove the claim for smooth,
compactly supported functions. For such functions one has
$(u,Hu)_{\rL^2}=\E_p(u)$ so that the positivity condition of
Lemma \ref{logarit} holds.
Lemma \ref{logarit} then implies that the logarithmic Sobolev
inequality \beq\label{log3} \int_X\dfrac{u^2}{\Vert
u\Vert_2^2}\log \dfrac{u^2}{\Vert u\Vert_2^2}\rd m\le
k_1\log\left(k_2\dfrac{\E_p(u)}{\Vert u\Vert_2^p}\right)\eeq holds
true for any $u\in\rL^2(X,m)$, where the positive constants
$k_1,k_2$ are those appearing in Theorem \ref{logarit}.

The thesis is then an immediate
consequence of Theorem 10.2 of \cite{BCLS}. In fact, it suffices
to use the known contraction properties of $\E$ discussed in \cite{BCLS} and
inequality (\ref{log3}) above.
\end{proof}

It is remarkable that the above logarithmic Sobolev inequality
has exactly the same form as the {\it Euclidean} one, proved first
in \cite{CG2} if $p<d$ and later on, with sharp constants, in
\cite{DD1}, \cite{G}. In particular the proportionality constant
in front of the r.h.s. of \eqref{log4} equals $d/p$ (with the present definition
of $d$) as in the Euclidean case.

We are now ready to state the main result of this
section: roughly speaking it says that, for the semigroups
considered here, hypercontractivity implies ultracontractivity, a
property which is clearly false in the linear case. Before stating the
Theorem we comment that, by the results of \cite{CG}, the
semigroup considered has a well defined version acting as a strongly
continuous contraction semigroup on all $\rL^p$ spaces,
$p\in[1,+\infty)$.

\begin{thm}\label{plap}
Let $\{T_t\}_{t\ge0}$ be the semigroup associated to the
generalized $p$-energy functional $\E$ given in (\ref{pl}), with
$p>2$. Assume that it is hypercontractive in the sense of
Definition \ref{iper}, with $\dot \alpha(0)<0$, $\dot r(0)>0$.
Then, for any $\varrho\in [q,+\infty]$ the following supercontractive and ultracontractive bounds hold true:
\begin{equation}\label{Main-Hyper}
\|T_tu\|_{\varrho}\leq C\frac{\|u\|_q^{\gamma}}{t^{\alpha}}
\end{equation}
for all $q\in\rL^q(M,m)$, where, for finite $\varrho$,
\begin{equation}
\alpha =\frac{d}{\varrho}\frac{\varrho -q}{pq+d(p-2)},\ \ \ \
\gamma =\frac{q}{\varrho}\frac{p\varrho +d(p-2)}{pq+d(p-2)}
\end{equation}
whereas, for $\varrho=\infty$:
\begin{equation}\label{alfagamma}
\alpha =\frac{d}{pq+d(p-2)},\ \ \ \ \gamma =\frac{pq}{pq+d(p-2)}.
\end{equation}
\end{thm}
The proof of the above Theorem can be done mimicking the discussion of \cite{BoG2}, since the proofs in that
paper did not depend either on the Euclidean setting discussed there or on the explicit, unweighted form of the generator, but only on the
validity of \eqref{log4} and on the fact that $\mathcal{E}_p$ is defined in terms of a suitable \it derivation. \rm

\section{Subordination of nonlinear semigroups}
A well known method for defining, in the linear setting, a
functional calculus for generators $A$ of strongly continuous
nonexpansive semigroup $\{T_t\}_{t\ge0}$ (relative to the class of
the so called Bernstein functions) is to define a new semigroup
$\{T_t^{(f)}\}_{t\ge0}$ called the Bochner subordinated semigroup
\cite{Bo}, and then to consider its generator, which in fact is a
possible definition of $f(A)$. We briefly recall this
construction.

We first consider a {\it convolution semigroup} of probability
measures $\{\mu_t\}_{t\ge0}$ on $[0,+\infty)$. By this we mean
that
\[
\left\{\begin{aligned} &\mu_t*\mu_s=\mu_{s+t}\ \forall s,t>0\\
&\mu_t\to\delta_0\ \ \ {\rm vaguely\ as\ }t\to0
\end{aligned}
\right.
\]
where $\delta_0$ is the Dirac delta at the origin. It is well
known (see \cite{J}), that there exists a function $f$ such that
\beq ({\mathcal L}\mu_t)(x)=e^{-tf(x)}\ \ \forall t>0,\eeq where
${\mathcal L}$ denotes the Laplace transform. Moreover $f$ is
well known to be a Bernstein function, i.e. a nonnegative
$C^\infty$ function on $(0,+\infty)$ with
\[
(-1)^kf^{(k)}(x)\le0\ \ \forall k\in{\mathbb N},x>0.
\]
Clearly any such $f$ cannot diverge faster then linearly as
$x\to+\infty$.

Then we may define, given a (nonlinear) strongly continuous
nonexpansive semigroup $\{T_t\}_{t\ge0}$ on a Hilbert space
$\rL^2(X,m)$ with generator $A$, the subordinated family
\beq\label{subord} S_tu:=\int_0^{+\infty}T_su\,\mu_t(\ds), \eeq
for all positive $t$ and all $u\in\rL^2$, provided the above
Bochner integral is finite. A sufficient condition for this to
hold is that $\{T_t\}_{t\ge0}$ is contractive. If $\{T_t\}_{t\ge0}$ is nonexpansive, a particularly important case since this property is satisfied by semigroups associated to convex and lower semicontinuous functionals, it suffices that one has in addition $T_t0=0$ for all $t$ to get contractivity as well.
More generally (see \cite{CG}) one
could assume that there exists a bounded orbit for the nonexpansive semigroup
$\{T_t\}_{t\ge0}$, a condition which then implies that {\it all}
orbits are bounded. In fact if $u$ has a bounded orbit and $v\in\rL^2$ is arbitrary, the nonexpansivity of the semigroup yields
\[
\|T_tv\|_2\le \|T_tv-T_tu\|_2+\|T_tu\|_2\le \|u-v\|_2+C
\]
for a suitable $C$ independent of $t$. We shall anyway {\it assume} in the sequel
without further comment that $S_tu$ is well defined for all
positive $t$ and all $u\in\rL^2$.

Then we define the operator \beq
A_fu:=\lim_{t\to0^+}\dfrac{u-S_tu}{t}=\lim_{t\to0}\int_0^{+\infty}T_su\dfrac{\delta_0-\mu_t}{t}(\rd
s) \eeq for all those $u\in\rL^2$ for which the limit exists in
$\rL^2$. Our aim will be to show for some particularly relevant
choices of $f$ that the limit exists for all $u\in D(A)$, to prove
that it is a monotone operator and to give an explicit formula for
it.

We shall use the notation
\[ \nu_t:=\dfrac{\delta_0-\mu_t}{t}.
\]
Then $\nu_t$ is a finite Radon measure on $[0,+\infty)$.

Hereafter, we shall also use the notation ${\mathcal S}([0,+\infty))$ to indicate the space of restrictions to $[0,+\infty)$ of functions belonging to the Schwartz space ${\mathcal S}({\mathbb R})$. Moreover ${\mathcal S}^\prime([0,+\infty))$ will denote the space of all tempered distributions on the real line whose support is contained in $[0,+\infty)$. The Laplace transform ${\mathcal L}u$ of an element $u\in{\mathcal S}^\prime([0,+\infty))$ is the analytic function in the open right half-plane $\mathfrak{Re}z>0$ given by ${\mathcal L}u(z)=\hat u(iz)$, where $\hat u$ denotes the Fourier transform of $u$. In the sequel we shall sometimes simply write $\mathcal{S}$ and ${\mathcal S}^\prime$ instead of ${\mathcal S}([0,+\infty))$ and ${\mathcal S}^\prime([0,+\infty))$ since no confusion can occur.

\begin{lem}
There exists a tempered distribution $\nu\in{\mathcal
S}^\prime([0,+\infty))$ such that ${\mathcal L}\nu=f$.
\end{lem}
\begin{proof}
A well known property of Bernstein functions shows that $f$ can
be written as \beq f(x)=a+bx+\int_{(0,+\infty)}(1-e^{-sx})\mu(\rd
s) \eeq where $a,b\ge0$ and $\mu$ is a nonnegative measure on
$(0,+\infty)$ such that \beq \int_{(0,+\infty)}
\dfrac{s}{1+s}\mu(\ds)<+\infty. \eeq The formula
\[
f(z)=a+bz+\int_{(0+\infty)}(1-e^{-sz})\mu(\ds)
\]
extends $f$ over ${\rm Re}\,z\ge0$ to an analytic function on
${\rm Re}\,z>0$. We claim that
\[
|f(z)|\le C(1+|z|)\ \ \forall z\ {\rm s.t.\ Re}\,z>0.
\]
Indeed, setting $z=x+iy$ for $x\ge0$:
\[
\begin{split}
&|1-e^{-sz}|^2=|1-e^{-sx}\cos(sy)-ie^{-sx}\sin(sy)|^2\\
&=(1-e^{-sx}\cos(sy))^2+e^{-2sx}\sin^2(sy)\\
&=1-2e^{-sx}\cos(sy)+e^{-2sx}\\
&\le 1+2e^{-sx}\left(\dfrac{s^2y^2}{2}-1\right)+e^{-2sx}\\
&=1-2e^{-sx}+e^{-2sx}+e^{-sx}s^2y^2\\
&\le(1-e^{-sx})^2+s^2y^2\le s^2(x^2+y^2)
\end{split}
\]
i.e. $|1-e^{-sz}|\le s|z|$ for all $s\ge0$, ${\rm Re}\,z\ge0$.
Since moreover $|1-e^{-sz}|\le 2$ for such $s,z$ and then
\[
\int_{(0,+\infty)}|1-e^{-sx}|\mu(\ds)\le
2\mu((1,+\infty))+|z|\int_{(0,1)}s\mu(\ds)
\]
(the latter integral being finite by the properties of $\mu$), the
claim is proved.

This implies that $f$ is the Laplace transform of a
tempered distribution $\nu\in{\mathcal
S}^\prime([0,+\infty)$ by \cite[page 306]{Scw}.
\end{proof}

\begin{lem}\label{esse}
With the above notations, the identification
$\nu=\lim_{t\to0}\nu_t$ holds true in the space ${\mathcal
S}^\prime([0,+\infty))$.
\end{lem}
\begin{proof}
By \cite[page 307, Remarque 1]{Scw} we know that a
net $\Lambda_t$ converges to
zero in ${\mathcal S}^\prime([0,+\infty))$ if:
\begin{itemize}
\item $({\mathcal L}\Lambda_t)(z)$ converges to zero uniformly
over the compact sets of the open right half-plane;
\item for any compact interval $[a,b]\subset(0,+\infty)$ there
exists a polynomial $p$ depending on $a,b$ such that
\[
|({\mathcal L}\Lambda_t)(z)|\le p(y)\ \ \ \forall
z=x+iy\in[a,b]\times{\mathbb R}
\]
for all $t$ sufficiently small.
\end{itemize}
We apply this result to the net $\Lambda_t=\nu_t-\nu$. In fact,
\[
({\mathcal
L}\Lambda_t)(z)=\dfrac{1-e^{-tf(z)}}{t}-f(z)=\dfrac{1-tf(z)-e^{-tf(z)}}{t}
\]
converges to zero uniformly over compact sets of ${\mathfrak Re}\,z>0$
since $f$ is continuous. Moreover, if $tx\le1$ then, setting $w=x+iy$:
\[
\begin{split}
&|1-tw-e^{-tw}|^2=|1-tx-ity-e^{-tx}(\cos(ty)-i\sin(ty))|^2\\
&=|(1-tx-e^{-tx}\cos(ty))+i(ty-e^{-tx}\sin(ty))|^2\\
&=(1-tx)^2-2(1-tx)e^{-tx}\cos(ty)+e^{-2tx}\cos^2(ty)\\
&+t^2y^2-2tye^{-tx}\sin(ty)+e^{-2tx}\sin^2(ty)\\
&\le(1-tx)^2+2(1-tx)e^{-tx}\left(\dfrac{t^2y^2}{2}-1\right)+e^{-2tx}+t^2y^2-2tye^{-tx}\sin(ty)\\
&=\left[(1-tx)^2-2(1-tx)e^{-tx}+e^{-2tx}\right]\\
&+t^2y^2(1-tx)e^{-tx}+t^2y^2-2tye^{-tx}\sin(ty)\\
&=(1-tx-e^{-tx})^2+t^2y^2(1-tx)e^{-tx}+t^2y^2-2tye^{-tx}\sin(ty)\\
&\le c_0+c_1\vert y\vert+c_2y^2
\end{split}
\]
for suitable $c_0,c_1,c_2\in{\mathbb R}$ depending on the compact
ranges of $x$ and $t$. Finally we notice that, setting $z=\varrho e^{i\vartheta}$, $\vartheta\in(-\pi,\pi]$ and defining $z^\alpha$ in the open right half-plane as $z^\alpha:=\varrho^\alpha e^{i\alpha\vartheta}$, we have $\mathfrak{Im}(z^\alpha)=\varrho^\alpha\sin(\alpha\vartheta)$ so that $|\mathfrak{Im}(z^\alpha)|\le\varrho^\alpha\le a_0+a_1|y|$.
\end{proof}
We shall now specialize to a special and particularly relevant
choice of the function $f$. Namely we shall consider the case
$f_\alpha(x)=x^\alpha$ for $\alpha\in(0,1)$. In this case it is
easy to write down an explicit formula for $\nu$.
\begin{lem}
The function $f_\alpha(x)=x^\alpha$ for $\alpha\in(0,1)$, $x\ge0$
is the Laplace transform of the tempered distribution
$\tau_\alpha$ given by \beq\label{tau}
\langle\tau_\alpha,\varphi\rangle:=\dfrac{\alpha}{\Gamma(1-\alpha)}\int_0^\infty\dfrac{\varphi(0)-\varphi(s)}{s^{1+\alpha}}\,\rd
s \eeq for all test function $\varphi$, where $\Gamma$ indicates
the Euler Gamma function. In particular the net $\nu_t$ converges
to $\tau_\alpha$ in ${\mathcal S}^\prime$ as $t\to 0$.
\end{lem}
\begin{proof}
By the definition of the Euler Gamma function:
\[
x^{\alpha-1}=\dfrac{1}{\Gamma(1-\alpha)}\int_0^{\infty}\dfrac{e^{-sx}}{s^{\alpha}}\,\ds,\ \ \forall x>0.
\]
Thus the function $x_+^{\alpha-1}$ is the Laplace transform of the tempered distribution $\sigma_\alpha(s)=\Gamma(1-\alpha)^{-1}s_+^{-\alpha}$. Integrating by parts one has $\tau_\alpha=\sigma_\alpha^\prime$. Thus
\[
{\mathcal L}(\tau_\alpha)(x)=x{\mathcal L}(\sigma_\alpha)(x)=x^\alpha\ \ \ \forall x>0.
\]
It is then immediate to check that
\[
\langle
(s^{-\alpha})^\prime,\varphi\rangle=\alpha\int_{0}^\infty\dfrac{\varphi(0)-\varphi(s)}{s^{1+\alpha}}\,\rd
s.
\]
\end{proof}
Since we aim at making $\tau_\alpha$ act on the function $s\mapsto
T_su$ for $u$ in the L$^2$-domain $D(A)$ of the generator $A$, we have to prove
that the convergence of $\nu_t$ to $\tau_\alpha$ takes place in a
stronger sense. To this end we introduce the space
\[
E=\left\{\psi\in C_b([0,+\infty))\  \rm s.t.\it\  \lim_{t\downarrow0}\frac{\psi(t)-\psi(0)}t\rm \ is\ finite\right\}.
\]
Since $\nu_t$ are finite measures on $[0,+\infty)$, $\nu_t(\psi)$ makes sense for all bounded continuous functions $\psi$. Notice also that the r.h.s. of formula \eqref{tau} still makes sense for all $\psi\in E$, thus defining a linear functional on $E$ still denoted by $\tau_\alpha$.
\begin{lem}
For all $\psi\in E$ one has
\beq
\lim_{t\downarrow0}\nu_t(\psi)=\frac{\alpha}{\Gamma(1-\alpha)}\int_0^{+\infty}\frac{\psi(0)-\psi(s)}{s^{1+\alpha}}\ds:=\tau_\alpha(\psi).
\eeq
\end{lem}
\begin{proof}
We denote by $g_t(x)$ the continuous function defining the density of $\mu_t$ with respect to the Lebesgue measure \cite{J}. It is a standard fact that $g_t(x)=t^{-1/\alpha}g_1(xt^{-1/\alpha})$ for all $x\ge0$, $t>0$. We shall write $g$ instead of $g_1$ from now on. We shall need a property of $g$, namely the fact that $g(x)\sim c_1x^{-1-\alpha}$ as $x\to+\infty$, where $c_1$ is a suitable positive constant, whose explicit value is known but will not be useful in the sequel \cite{P}. It follows that $\int_0^x yg(y)\rd y\sim c_2x^{1-\alpha}$ as $x\to+\infty$. This latter fact can also be proved directly by using only the expression of the Laplace transform, noticing that $\frac{\rd}{\rd x}{\mathcal L}(\mu_t)(x)=-{\mathcal L}(x\mu_t)(x)$ and using a Tauberian Theorem.

Take now $\psi\in E, \varphi\in {\mathcal S}$ and write
\[
\begin{split}
\int_0^{+\infty}\nu_t(\rd x)[\varphi(x)-\psi(x)]&=\frac1t\left[\varphi(0)-\psi(0)-\int_0^{+\infty}\hskip-10pt g_t(x)[\varphi(x)-\psi(x)]\rd x\right]\\
&=\frac1t\int_0^{+\infty}\hskip-10pt g_t(x)[(\varphi(0)-\varphi(x))-(\psi(0)-\psi(x))]\rd x\\
&=\frac1t\int_0^{1} g_t(x)[(\varphi(0)-\varphi(x))-(\psi(0)-\psi(x))]\rd x\\
&+\frac1t\int_1^{+\infty}\hskip-10pt g_t(x)[(\varphi(0)-\varphi(x))-(\psi(0)-\psi(x))]\rd x.
\end{split}
\]
Define the seminorm $p:E\to[0,+\infty)$ by:
\[
p(\psi):=\sup_{x>0}\left|\frac{\psi(x)-\psi(0)}{x\wedge x^{\alpha/2}}\right|.
\]
Then we have, for a suitable constant $C_1>0$:
\[
\begin{split}
&\left|\frac1t\int_0^{1} g_t(x)[(\varphi(0)-\varphi(x))-(\psi(0)-\psi(x))]\rd x\right|\\
&\le\left(\sup_{x\in(0,1)}\left|\frac{(\varphi(0)-\varphi(x))-(\psi(0)-\psi(x))}{x}\right|\right)\frac1t\int_0^1g_t(x)x\rd x\\
&\le p(\varphi-\psi)\frac1t\int_0^1g_t(x)x\rd x\\
&=p(\varphi-\psi)\frac1{t^{1+1/\alpha}}\int_0^1g(xt^{-1/\alpha})x\rd x\\
&=p(\varphi-\psi)\frac1{t^{1-1/\alpha}}\int_0^{t^{-1/\alpha}}\hskip-15pt g(y)y\rd y\\
&\le C_1p(\varphi-\psi).
\end{split}
\]
for all $t$ sufficiently small. Similarly, for a suitable constant $C_2>0$:
\[
\begin{split}
&\left|\frac1t\int_1^{+\infty}\hskip-10pt g_t(x)[(\varphi(0)-\varphi(x))-(\psi(0)-\psi(x))]\rd x\right|\\
&\le\left(\sup_{x\ge1}\left|\frac{(\varphi(0)-\varphi(x))-(\psi(0)-\psi(x))}{x^{\alpha/2}}\right|\right)\frac1t\int_1^{+\infty}\hskip-10pt
g_t(x)x^{\alpha/2}\rd x\\
&\le p(\varphi-\psi)\frac1t\int_1^{+\infty}\hskip-10pt g_t(x)x^{\alpha/2}\rd x\\
&=p(\varphi-\psi)\frac1{t^{1+1/\alpha}}\int_1^{+\infty}\hskip-10pt g(xt^{-1/\alpha})x^{\alpha/2}\rd x\\
&=p(\varphi-\psi)\frac1{t^{1/2}}\int_{t^{-1/\alpha}}^{+\infty}\hskip-10pt g(y)y^{\alpha/2}\rd y\\
&\le C_2p(\varphi-\psi),
\end{split}
\]
again for all $t$ sufficiently small so that, for all such values of $t$:
\beq
\left|\int_0^{+\infty}\nu_t(\rd x)[\varphi(x)-\psi(x)]\right|\le Cp(\varphi-\psi)
\eeq
for a suitable positive constant $C$. Proceeding in a very similar manner allows to prove the inequality
\beq
\left|\tau_\alpha(\varphi(x)-\psi(x))\right|\le Cp(\varphi-\psi).
\eeq
Now we notice that
\[\begin{split}
|\nu_t(\psi)-\tau_\alpha(\psi)|&\le |\nu_t(\varphi)-\tau_\alpha(\varphi)|+|\nu_t(\psi-\varphi)|+|\tau_\alpha(\psi-\varphi)|\\
&\le |\nu_t(\varphi)-\tau_\alpha(\varphi)|+Cp(\psi-\varphi)
\end{split}
\]
so that, by Lemma \ref{esse}:
\[
\lim_{t\downarrow0}|\nu_t(\psi)-\tau_\alpha(\psi)|\le Cp(\psi-\varphi).
\]
Therefore we get $\limsup_{t\downarrow0}|\nu_t(\psi)-\tau_\alpha(\psi)|=0$ for all $\psi\in E$ provided we prove that for all positive $\varepsilon$ there exists a function $\varphi\in{\mathcal S}$ with $p(\psi-\varphi)\le\varepsilon$. It suffices to consider only the case $\psi(0)=0$ so that $\varphi(0)=0$ can be assumed as well. Then, let $h\in C^{\infty}([0,+\infty))$ be such that $0<h(x)\le x\wedge x^{\alpha/2}$ for all $x>0$ and $h(x)=x\wedge x^{\alpha}$ for $x\in[0,1/2]\cup[2,+\infty)$. We have $p(\psi-\varphi)\le q(\psi-\varphi):=\sup_{x\ge0}h(x)^{-1}|\psi(x)-\varphi(x)|$ so that it suffices to show that $q(\psi-\varphi)$ is small for $\varphi\in\mathcal{S}$ chosen appropriately. Notice that $h(x)^{-1}\psi(x)$ can be extended to a function belonging to $C_0([0,+\infty))$ since $\psi(0)=0$ and the right derivative of $\psi$ at $t=0$ exists. In turn this follows from the fact that $\mathcal{S}([0,+\infty))$ is dense in $C_0([0,+\infty))$ (the space of continuous functions $g$ on $[0,+\infty)$ such that $\lim_{x\to+\infty}g(x)=0$) in the uniform topology so that, if $f\in \mathcal{S}$ is a function close to the function $h(x)^{-1}\psi(x)\in C_0([0,+\infty))$ in the uniform topology, the function $\varphi=fh$ belongs to $\mathcal{S}$ and is close to $\psi$ in the topology associated to the norm $q$.
\end{proof}
We have now all the ingredients to prove the following result.
\begin{thm}\label{potenza}
Let $\{T_t\}_{t\ge0}$ be a (nonlinear) strongly continuous
nonexpansive semigroup on a Hilbert space $\rL^2(X,m)$, with
generator $A$. Suppose that $\{T_t\}_{t\ge0}$ has a bounded orbit.
Let $\{\mu^{(\alpha)}_t:t\ge0\}$ be the convolution semigroup
associated to the Bernstein function $f_\alpha(x)=x^\alpha$ for
$x\ge0$, where $\alpha\in(0,1]$. Let finally $S_t$ be the
subordination of $\{T_t\}_{t\ge0}$ defined by \beq\label{subord2}
S_tu:=\int_0^{+\infty}T_su\,\mu^{(\alpha)}_t(\ds),
\eeq for all $t\ge0$ and all $u\in\rL^2$. Then the right
derivative of $S_tu$ exists at $t=0$ for all $u\in D(A)$ and,
denoting it by $A^\alpha u$, the formula \beq
\label{generatore}A^\alpha
u=\dfrac{\alpha}{\Gamma(1-\alpha)}\int_0^\infty\dfrac{u-T_su}{s^{1+\alpha}}\rd
s \eeq holds for any $u\in D(A)$. Moreover the operator
$A^\alpha:D(A)\to\rL^2$ is monotone, so that it admits a
maximally monotone extension which then defines a (nonlinear)
strongly continuous nonexpansive semigroup
$\{T_t^\alpha\}_{t\ge0}$ on $\rL^2$, called the
$\alpha$-subordinated semigroup.
\end{thm}
\begin{proof}
It suffices to notice that the assumption that $\{T_t\}_{t\ge0}$
has a bounded orbit and the nonexpansivity of $\{T_t\}_{t\ge0}$
imply that all orbits are bounded, so that $S_t$ is
well defined and that, by \cite[Theorem 3.1]{B}, the map $s\to T_s u$
is Lipschitz continuous for any fixed $u\in D(A)$ and has a right derivative at $s=0$ for any fixed $u\in D(A)$.

The monotonicity of $A^\alpha$ is a consequence of:
\[
\begin{split}
&(A^\alpha u-A^\alpha v,u-v)=c\int_0^{+\infty}\dfrac{\rd
s}{s^{1+\alpha}}(u-v-(T_su-T_sv),u-v))\\
&=c\int_0^{+\infty}\dfrac{\rd
s}{s^{1+\alpha}}\left[\|u-v\|_2^2-(T_su-T_sv,u-v)\right]\\
\end{split}
\]
and of the fact that
\[
(T_su-T_sv,u-v)\le |(T_su-T_sv,u-v)|\le
\|T_su-T_sv\|_2\|u-v\|_2\le \|u-v\|_2^2
\]
by the nonexpansivity of $\{T_s\}$.

The latter statement follows by \cite[Corollary 2.1]{B}.
\end{proof}
\begin{cor}
The inequality
\[
\|A^\alpha u\|_2\le \dfrac{\|A_\circ
u\|^\alpha_2\sup_{s\ge0}\|u-T_su\|_2^{1-\alpha}}{(1-\alpha)\Gamma(1-\alpha)}
\]
where $A_\circ$ is the principal section of $A$ in the sense of
\cite{B}.
\end{cor}
\begin{proof}
One has, for all positive $x$:
\[
\begin{split}
&\dfrac{\Gamma(1-\alpha)}{\alpha}\|A^\alpha u\|_2\le
\int_0^\infty\frac{\|u-T_su\|_2}{s^{1+\alpha}}\rd
s\\
&=\int_0^x\frac{\|u-T_su\|_2}{s^{1+\alpha}}\ds+\int_x^\infty\frac{\|u-T_su\|_2}{s^{1+\alpha}}\ds\\
&\le \|A_\circ u\|_2\int_0^xs^{-\alpha}\rd
s+\sup_{s\ge0}\|u-T_su\|_2\int_x^\infty s^{-(1+\alpha)}\rd s\\
&=\|A_\circ
u\|_2\dfrac{x^{1-\alpha}}{1-\alpha}+\sup_{s\ge0}\|u-T_su\|_2\dfrac{x^{-\alpha}}{\alpha}.
\end{split}
\]
where we have used the fact (see \cite[Theorem 3.1, item (2)]{B}) that
$\|\rd T_su/\ds\|_{\rL^\infty((0,+\infty);\rL^2)}\le \|A_\circ u
\|_2$ for all $u\in D(A)$. Optimizing over $x>0$ we get the
assertion.
\end{proof}
\section{Nash estimates for generators of subordinated semigroups.}
Our aim in this section will be to use the above construction of
nonlinear subordinated semigroups in the specific setting of
section 4, i.e. when $\{T_t\}_{t\ge0}$ is the nonlinear semigroup
generated by the generalized $p$-Laplacian introduced in such
section. We shall show that, when such semigroup is
hypercontractive, its subordination $\{T_t^\alpha\}_{t\ge0}$ is
hypercontractive as well, so that as a consequence we shall note
that Nash-type inequalities hold for $A^\alpha$ too.

We start with the following Lemma.
\begin{lem}
Let $\{\mu_t:t\ge0\}$ be the convolution semigroup
associated to the Bernstein function $f_\alpha(x)=x^\alpha$ for
$x\ge0$, where $\alpha\in(0,1]$ is fixed. Then, for any $\beta>0$ and all
$t>0$ the integral $\int_0^{+\infty}\mu_t(\ds)s^{-\beta}$ is
finite.
\end{lem}
\begin{proof}
It is clear, since each $\mu_t$ is a probability measure, that it
suffices to prove the claim for $\beta\ge1$. For such $\beta$ an
elementary induction argument and the fact that $e^{-tx^\alpha}$
is the Laplace transform of $\mu_t$ show that, if $[\cdot]$
denotes the integer part of a real number:
\[
\int_0^{+\infty}\rd
x\,x^{\beta-1}e^{-tx^\alpha}=\left(\int_0^{+\infty}\rd
y\,y^{\beta-[\beta]}e^{-y}\right)\int_0^{+\infty}\mu_t(\rd
s)\dfrac{[\beta-1]!}{s^\beta}
\]
so that in particular the latter integral in the r.h.s. is finite.
\end{proof}
As a consequence of the above Lemma and of the results of section
2 we have:
\begin{lem}\label{boh}
Let $\{T_t\}_{t\ge0}$ be the semigroup associated to the
$p$-energy functional $\E$ given in (\ref{pl}), with $p>2$.
Assume that it is hypercontractive in the sense of Definition
\ref{iper}, with $\dot \alpha(0)<0$, $\dot r(0)>0$. Then the
subordinated family $\{S_t\}$ satisfies, for all
$\varrho>2$, the supercontractive bound
\[
\|S_tu\|_\varrho\le
k_\varrho(t)\|u\|_2^{\gamma(\varrho)}
\]
where
\[
\gamma(\varrho) =\frac{2}{\varrho}\frac{p\varrho
+d(p-2)}{2p+d(p-2)}.
\]
\end{lem}
\begin{proof}
This is an immediate consequence of the bound
\[
\|S_tu\|_\varrho\le\int_0^{+\infty}\|T_su\|_\varrho\mu_t(\ds)\le
C\|u\|_2^{\gamma(\varrho)}\int_0^{+\infty}s^{-\beta}\mu_t(\ds)
\]
for a suitable positive $\beta$, valid because of the results of
Theorem \ref{plap}, and of the above Lemma.
\end{proof}

The results of section 2 then can be used directly to prove
logarithmic Sobolev inequalities for the right derivative at $t=0$
of $S_t$, although such map does not give rise to a semigroup. In
fact we made no use of the semigroup property there. Noticing in addition
that the explicit expression of $\gamma$ shows immediately that
$\varrho\gamma(\varrho)>2$ for any $\varrho>2$, one therefore has, proceeding as in the proof of Theorem \ref{GNI1}:
\begin{lem}
Let, for $\alpha\in(0,1)$, $A^\alpha$ be the maximally monotone
operator (see Theorem \ref{potenza}) associated to the
right derivative at $t=0$ of the subordinated family $S_tu$ defined in
the previous Lemma. Then, under the assumptions of such Lemma, the inequality
\begin{equation}
\dfrac{\varrho-2}{\varrho}\int_Xu^2\log |u|\rd
m-\dfrac{\varrho\gamma(\varrho)-2}{\varrho\gamma(\varrho)}\Vert
u\Vert_{{2}}^2\log \Vert u\Vert_{{2}}\le (u,A^\alpha
u)_{\rL^2}+C\Vert u\Vert_{{2}}^2
\end{equation}
holds true for each $u$ in $D(A)$ and
$\varrho>2$.
\end{lem}
It may be useful to recall again that $D(A_\alpha)\supset D(A)$.

To proceed further we now prove subhomogeneity for the functional
$(u,A^\alpha u)$. In fact:
\begin{lem}
With the above notations and assumptions, the operator $A^\alpha$
enjoys the following property: for any $u\in D(A)$ and for all
positive $\lambda$ one has \beq A^\alpha (\lambda
u)=\lambda^{1+\alpha(p-2)}A^\alpha u. \eeq
\end{lem}
\begin{proof}
We use the explicit expression for $A_\alpha$ when acting on
functions belonging to the domain of $A$ and the fact that
\[
T_s(\lambda u)=\lambda T_{\lambda^{p-2}s}u,
\]
a property which can be verified directly from the differential
equation satisfied by $T_s$. In fact:
\[
\begin{split}
&A^\alpha(\lambda u)=c\int_0^{+\infty}\dfrac{T_s(\lambda
u)-\lambda u}{s^{1+\alpha}}\ds=c\lambda\int_0^{+\infty}\dfrac{T_{\lambda^{p-2}s}u-u}{s^{1+\alpha}}\ds\\
&=c\lambda^{1+\alpha(p-2)}\int_0^{+\infty}\dfrac{T_{s}u-u}{s^{1+\alpha}}\rd
s=\lambda^{1+\alpha(p-2)}A^\alpha u.
\end{split}
\]
\end{proof}

The following Proposition then follows along the same line of proof given in Theorem \ref{GNI1}.

\begin{prop}\label{logsobsub}Under the assumptions of Lemma \ref{boh}, the logarithmic Sobolev inequality
\[
\int_X\dfrac{u^2}{\Vert u\Vert_2^2}\log \dfrac{u^2}{\Vert
u\Vert_2^2}\rd m\le A_1\log\left(A_2\dfrac{(u,A^\alpha
u)_{\rL^2}}{\Vert u\Vert_2^{2+\alpha(p-2)}}\right)
\]
holds for any $u\in D(A)$, with
$A_1=2(1-\gamma(\varrho))/([\gamma(\varrho)(p-2)(\varrho-2)]$ and
$A_2$ a suitable positive constant.
\end{prop}
\begin{proof}
First we notice that $\gamma(\varrho)<1$ for all $\varrho>2$ and
that, as already stated,  $\varrho\gamma(\varrho)>2$ for all such
$\varrho$. The fact that $(u,A^\alpha u)$ is nonnegative follows
from the explicit expression (\ref{generatore}) of $A^\alpha$ and
from the fact that $T_s$ is {\it contractive} in $\rL^2$ because
$T_s0=0$ for all $s$ as $A0=0$. Finally, the functional
$(u,A^\alpha u)$ is, by the previous Lemma, homogeneous of degree
$2+\alpha(p-2)$ which is then always strictly larger than two.

\end{proof}

We are ready to state the final result of this paper, whose proof
is the same given in Theorem \ref{GNI1}.
\begin{thm}\label{sub}(Nash-type inequalities for generators subordinated to the $p$-Laplacian).
Let $(M,g)$ be a complete Riemannian manifold, whose Riemannian
gradient is indicated by $\nabla$. Let $m$ be a $\sigma$-finite
nonnegative measure on $M$, and assume that $\nabla$ is a closed
operator from L$^2(M,m)$ to L$^2(TM,m)$.Consider the strongly
continuous contraction semigroup $\{T_t\}_{t\ge0}$ associated to
the subgradient of the convex l.s.c. functional given by \beq
\E_p(u):=\int_M|\nabla u|^p\,{\rm d}m. \eeq Assume that it is
hypercontractive in the sense of Definition \ref{iper}, with $\dot
\alpha(0)<0$, $\dot r(0)>0$. Let $A^\alpha$ be the generator of
the subordinated semigroup associated to the Bernstein function
$f(x)=x^\alpha$ with $\alpha\in(0,1)$. Then, for any
$\vartheta\in(0,1)$ the inequality \beq \|f\|_2\le c(A^\alpha
f,f)_{{\rm
L}^2}^{\frac{A_1(1-\vartheta)}{\vartheta+pA_1(1-\vartheta)}}\|f\|_{2\vartheta}^{\frac{\vartheta}{\vartheta+pA_1
(1-\vartheta)}} \eeq holds true for a suitable constant $c$ and
for all $f\in{\rm Dom}\,A\cap{\rm L}^{2\vartheta}$, where $A_1>0$
is the constant appearing in Proposition \ref{logsobsub}. In
particular the Nash inequality \beq \|f\|_2\le c(A^\alpha
f,f)_{{\rm L}^2}^{\frac{A_1}{1+pA_1}}\|f\|_1^{\frac{1}{1+pA_1}}
\eeq holds true for all $f\in{\rm Dom}\,A\cap{\rm L}^1$.
\end{thm}
\begin{rem}
An elementary, although tedious, calculation shows that the
dimension $d_\alpha$ of the subordinated semigroup is, if $d$ is the
dimension of the original semigroup (in the sense given in the statement of Corollary \eqref{plapl}),
\[
d_\alpha=\dfrac{d[2+\alpha(p-2)]}{\alpha p}.
\]
\end{rem}

\bibliographystyle{amsplain}

\end{document}